\documentclass[12pt,a4paper,reqno]{amsart}
\usepackage[english]{babel}
\usepackage[applemac]{inputenc}
\usepackage[T1]{fontenc}
\usepackage{palatino}
\usepackage{amsmath}
\usepackage{amssymb}
\usepackage{amsthm}
\usepackage{amsfonts}
\usepackage{graphicx}

\usepackage{amsrefs}
\usepackage[colorlinks = true, citecolor = blue]{hyperref}
\title{Nonnegative kernels and $1$-rectifiability in the Heisenberg group}
\address{University of Connecticut, Department of Mathematics}
\address{University of Chicago, Department of Mathematics}
\author{Vasileios Chousionis}
\author{Sean Li}
\thanks{VC was supported by the Academy of Finland through the grant \emph{Geometric harmonic analysis}, grant number 267047. S.L. is supported by NSF grant DMS-1600804. }
\email{vasileios.chousionis@uconn.edu}
\email{seanli@math.uchicago.edu}

\newcommand{\ve}{\varepsilon}
\renewcommand{\H}{\mathbb{H}}
\newcommand{\R}{\mathbb{R}}
\newcommand{\Rn}{\mathbb{R}^n}

\newcommand{\N}{\mathbb{N}}

\newcommand{\C}{\mathbb{C}}
\newcommand{\Z}{\mathbb{Z}}

\renewcommand{\L}{\mathcal{L}}

\newcommand{\Hd}{\mathcal{H}}

\newcommand{\ha}{\mathcal{H}}

\newcommand{\diam}{\operatorname{diam}}

\newcommand{\He}{\mathbb{H}}

\newcommand{\stm}{\setminus}
\newcommand{\ra}{\rightarrow}

\numberwithin{equation}{section}

\newtheorem{thm}{Theorem}[section]

\newtheorem{lemma}[thm]{Lemma}
\newtheorem{corollary}[thm]{Corollary}
\theoremstyle{definition}

\newtheorem{proposition}[thm]{Proposition}
\theoremstyle{definition}

\newtheorem{df}[thm]{Definition}
\theoremstyle{definition}

\theoremstyle{definition}

\newtheorem{remark}[thm]{Remark}

\begin{document}

\begin{abstract} Let $E$ be an $1$-Ahlfors regular subset of the Heisenberg group $\H$. We prove that there exists a $-1$-homogeneous kernel $K_1$ such that if $E$ is contained in a $1$-regular curve the corresponding singular integral is bounded in $L^2(E)$.  Conversely, we prove that there exists another $-1$-homogeneous  kernel $K_2$, such that the $L^2(E)$-boundedness of its corresponding singular integral implies that $E$ is contained in an $1$-regular curve. These are the first non-Euclidean examples of kernels with such properties. Both $K_1$ and $K_2$ are weighted versions of the Riesz kernel corresponding to the vertical component of $\H$. Unlike the Euclidean case, where all known kernels related to rectifiability are antisymmetric, the kernels $K_1$ and $K_2$ are even and nonnegative.
\end{abstract}

\maketitle

\section{Introduction}

One of the standard topics in classical harmonic analysis is the study of singular integral operators 
(SIOs) of the form
$$Tf\,(x)=\int \frac{\Omega(x-y)}{|x-y|^{n}}f(y)\, d \L^n(y)$$
where $\Omega$ is a $0$-homogeneous function  and $\L^n$ is the Lebesgue measure in $\Rn$, see e.g. \cite{St}. A considerable amount of research has been devoted to such SIO's, and nowdays they are well understood. On the other hand if the singular integral is defined on lower dimensional measures the situation is much more complicated even when one considers the simplest of kernels. 

As an example the reader should think of the Cauchy transform $$C_E f\,(z)=\int_E \frac{f(w)}{z-w} d \ha^1 (w), \, E \subset \C,$$  where $\ha^1$ denotes the   $1$-dimensional Hausdorff measure in the complex plane. Two questions  arise naturally. For which sets $E$ is $C_E$ bounded in $L^2(E)$? And, if  $C_E$ is bounded in $L^2(E)$ what does this imply about $E$? Here $L^2(E)$-boundedness means that there exists a constant $C>0$ such that the truncated operator $$C_{E}^{\ve}f\,(z)=\int_{E \stm B(z,\ve)} \frac{f(w)}{z-w} d \ha^1 (w)$$ satisfies $\|C_{E}^{\ve}f\|_{L^2(\ha^1 |_E)} \leq C \|f\|_{L^2(\ha^1 |_E)}$ for all $f \in L^2(\ha^1 |_E)$. It turns out that the $L^2(E)$-boundedness of the Cauchy transform depends crucially on the geometric structure of $E$. 

The problem of exploring this relation has a long history and it is deeply related to rectifiability and analytic capacity; we refer to the recent book of Tolsa \cite{tolsabook} for an extensive treatment. One of the landmarks in the field was the characterization of the $1$-(Ahlfors-David)-regular sets $E$ on which the Cauchy transform is bounded in $L^2(E)$. Recall that  an $\ha^{1}$-measurable set $E$ is \emph{
$1$-(Ahlfors-David)-regular}, if there exists a constant $1\leq C
<\infty$, such that
\begin{displaymath}
C^{-1} r \leq \mathcal{H}^1(B(x,r)\cap E) \leq C r
\end{displaymath}
for all  $x\in E$, and $0<r\leq \mathrm{diam}(E)$. It turns out that that if $E$ is $1$-regular the Cauchy transform $C_E$ is bounded in $L^2(E)$ if and only if $E$ is contained in an $1$-regular curve. The sufficient condition is due to David \cite{Dsuff} and it even holds for more general smooth antisymmetric kernels. The necessary condition is due to Mattila, Melnikov and Verdera \cite{MMV}. It is a remarkable fact that their proof depends crucially on a special subtle positivity property of the Cauchy kernel related to an old notion of curvature named after Menger, see e.g., \cite{MeV} or \cite{MMV}. We also note that the above characterization also holds for the SIOs associated to the coordinate parts of the Cauchy kernel.

Very few things are known for the action of SIOs associated with other $-1$-homogeneous, $1$-dimensional Calder\'on-Zygmund kernels (see Section \ref{sec:prelim} for the exact definition) on $1$-AD regular sets in the complex plane. Call a kernel ``good'' if its associated SIO is bounded on $L^2(E)$ if and only if $E$ is contained in an $1$-regular curve.  It is noteworthy that all known good or bad kernels are related to the kernels
\begin{equation*}
\label{kn}
k_{n}(z)=\frac{x^{2n-1}}{|z|^{2n}}, \quad z =(x,y) \in \C\setminus \{0\}, n \in \N.
\end{equation*} 
Observe that $k_1$ is a good kernel as it is the $x$-coordinate of the Cauchy kernel, see \cite{MMV}. It was shown in \cite{CMPT} that  the kernels $k_n, n>1,$ are good as well, and these were the first non-trivial examples of good kernels not directly related with the Cauchy kernel. Now let 
$$\kappa_t(z)=k_2(z)+t\cdot k_1(z), t \in \R.$$ It follows by \cite{CMPT} and \cite{MMV} that $\kappa_t$ is good for $t>0$. Recently Chunaev \cite{chun} showed that $\kappa_t$ is good for $t \leq-2$ and Chunaev, Mateu and Tolsa \cite{chumt} proved that $\kappa_t$ is good for $t\in (-2,-\sqrt{2})$. For $t=-1$ and $t=-3/4$ there exist intricate examples of sets $E$, due to Huovinen \cite{huo} and Jaye and Nazarov \cite{jnaz} respectively, which show that the $L^2(E)$ boundedness of the SIO associated to $\kappa_{-1}$ and $\kappa_{-3/4}$ does not imply rectifiability for $E$.

Notice that all the kernels mentioned so far are odd and this is very reasonable. Consider for example an $1$-dimensional Calder\'on-Zygmund kernel $k: \R \times \R \stm\{x=y\} \ra \R^+$ which is not locally integrable along the diagonal. Think for example $k(x,y)=|x-y|^{-1}. $ Then $\int_I k(x,y) dy=\infty$ for all open intervals $I \subset \R$. It becomes evident that if one wishes to define a SIO which makes sense on lines and other ``nice'' $1$-dimensional objects, depends crucially on the cancellation properties of the kernel. Surprisingly in the Heisenberg group $\H$ the situation is very different. 

The \emph{Heisenberg group} $\He$ is $\mathbb{R}^3$
endowed with the group law
\begin{equation}\label{eq:Heis}
p\cdot q =(x+x',y+y',z+z'+(xy'-yx')/2)
\end{equation}
for $p=(x,y,t),q=(x',y',t') \in \mathbb{R}^3$.  We use the following metric on $\He$:
\begin{displaymath}
d_\H:\He\times \He\to [0,\infty),\quad
d_\H(p,q):= N(q^{-1} \cdot p),
\end{displaymath}
where $N: \H \ra [0,\infty)$ is the  \emph{Kor\'anyi norm} in $\H$,
\begin{align*}
  N(x,y,z) := ((x^2 + y^2)^2 + z^2)^{1/4}.
\end{align*}
We also let
\begin{align*}
  NH(x,y,z) = |z|^{1/2},
\end{align*}
where $NH$ stands for non-horizontal.  Note that
\begin{align*}
  d_\H(x,y) = (|\pi(x) - \pi(y)|^4 + NH(x^{-1}y)^4)^{1/4}.
\end{align*}	
We also remark that the metric $d_{\mathbb{H}}$ is  homogeneous with respect to the dilations
\begin{displaymath}
\delta_r:\He \to \He,\quad \delta_r((x,y,z))
=(rx,ry,r^2 z),\quad (r>0).
\end{displaymath}
Finally let $\Omega: \H \stm \{0\} \ra [0, \infty)$,
\begin{equation}
\label{omega}
\Omega (p)=\frac{NH(p)}{N(p)}
\end{equation}
and notice that $\Omega$ is $0$-homogeneous as $\Omega(\delta_r(p))= \Omega(p)$ for all $r>0$.

In our first main theorem we prove that, in contrast to the Euclidean case, there exists a nonnegative, $-1$ homogeneous, Calder\'on-Zygmund kernel which is bounded in $L^2(E)$ for every $1$-regular set $E$  which is contained in a $1$-regular curve. We warn the reader that  from now on $\ha^1$ will denote the $1$-dimensional Hausdorff measure in $(\H, d_\H)$.

\begin{thm}
\label{necesthm}
Let $K_1 : \H \stm \{0\} \ra [0,\infty)$ defined by
$$K_1(p)=\frac{\Omega(p)^8}{N(p)},$$
and let $E$ be a $1$-regular set $E$  which is contained in a $1$-regular curve.  Then the corresponding truncated singular integrals
$$T^\ve_1f\,(p)=\int_{E \stm B_{\H}(p,\ve)} K_1(q^{-1} \cdot p) f(q) \,d \ha^1(q)$$
are uniformly bounded in $L^2(E)$.
\end{thm}

We define the principal value of $f$ at $p$ to be
\begin{align*}
    \mathrm{p.v.}T_1f(p) = \lim_{\ve \to 0} T^\ve_1(f)(p),
\end{align*}
when the limit exists.  Because the kernel is positive, we will be able to use Theorem \ref{necesthm} to easily show that the principal value operator is bounded in $L^2$.

\begin{corollary}
\label{pvcor}
If $f \in L^2(E)$, then p.v.$T_1f(x)$ exists almost everywhere and is in $L^2(E)$.  Moreover, we have that there exists a constant $C > 0$ so that 
$$\|\mathrm{p.v.}T_1f\|_{L^2(E)} \leq C \|f\|_{L^2(E)}, \qquad \forall f \in L^2(E).$$
\end{corollary}

Let us quickly give an intuition behind why one would expect a positive kernel like $NH(x)^m/N(x)^{m+1}$ to be bounded on Lipschitz curves.  Rademacher's theorem says that Lipschitz curves in $\R^n$ infinitesimally resemble affine  lines and antisymmetric kernels cancel on affine lines.  This is essentially what controls the singularity.  In the Heisenberg setting, a Rademacher-type theorem by Pansu \cite{Pansu} says that Lipschitz curves infinitesimally resemble {\it horizontal} lines and $NH$ is $0$ on horizontal lines.  Thus, we again have control over the singularity.


Some heuristic motivation comes from the fact that the positive Riesz kernel $\frac{|z|}{(x^2 + y^2 + z^2)^{3/2}}$ defines a SIO which is trivially bounded in $\R^3$ for curves in the $xy$-plane.  In this case, however, boundedness of this SIO tells us nothing about the regularity of the $xy$ curve.  An analogous concern in the Heisenberg group would be whether boundedness of kernels of the form $NH(z)^p/N(z)^{p+1}$ imply anything about the regularity of the sets if the vertical direction is ``orthogonal'' to Lipschitz curves.  While we do not know if boundedness of the kernel of Theorem \ref{necesthm} says anything about regularity, our next result shows that there exists some $p$ for which these vertical Riesz kernels do:

\begin{thm}
\label{suffthm}
Let $K_2: \H \stm \{0\} \ra [0,\infty)$ defined by
$$K_2(p)=\frac{\Omega(p)^2}{N(p)},$$
and let $E$ be a $1$-regular set $E$. If the corresponding truncated singular integrals
$$T^\ve_2 f\,(p)=\int_{E \stm B_{\H}(p,\ve)} K_2(q^{-1} \cdot p) f(q) d \ha^1(q)$$ are uniformly bounded in $L^2(E)$ then $E$  is contained in a $1$-AD regular curve.
\end{thm}

Thus, the vertical fluctuations of a Lipschitz curve in $\H$ contain enough information to determine regularity.

The following question arises naturally from Theorems \ref{necesthm} and \ref{suffthm}. \emph{Does there exist some $m \in \N$ such that any $1$-regular set $E$ is contained in some $1$-regular curve if and only if the operators 
$$T^\ve f\,(p)=\int_E\frac{ \Omega(q^{-1}\cdot p)^m}{N(q^{-1}\cdot p)}f(q)\, d\ha^1(q)$$
are uniformly bounded in $L^2(E)$?}
The methods developed in this paper do not allow us to answer this question, partly because our proof for Theorem \ref{necesthm} seems to require a large power for $\Omega(p)$. This is essential because we are using a positive kernel and so are not able to use antisymmetry to gain additional control from the bilinearity as is commonly used in these types of arguments (for example, see section 6.2 of \cite{Tolsa}).  The proof of Theorem \ref{suffthm} uses delicate estimates regarding the Kor\'anyi norm and is also not likely to be improved without a major change in the proof structure.

A motivation for the geometric study of SIOs in $\Rn$ is their significance in PDE and potential theory. In particular the $d$-dimensional Riesz transforms (the SIOs associated to the kernels $x/|x|^{d+1}$) for $d=1$ and $d=n-1$  play a crucial role in the geometric characterization of \emph{removable sets} for bounded analytic functions and Lipschitz harmonic functions. Landmark contributions by David \cite{Da3}, David and  Mattila \cite{DM}, and  Nazarov,  Tolsa and Volberg \cite{NToV1}, \cite{NToV2}, established that these removable sets coincide with the purely $(n-1)$-unrectifiable sets in $\Rn$, i.e. the sets which intersect every $(n-1)$-dimensional Lipschitz graph in a set of vanishing $(n-1)$-dimensional Hausdorff measure.  For an excellent review of the topic and its connections to non-homogeneous harmonic analysis we refer the reader to the survey \cite{EiV} by Eiderman and Volberg.

The same motivation exists in several non-commutative Lie groups as well. For example the problem of characterizing removable sets for Lipschitz harmonic functions has a natural analogue in Carnot groups. In that case the  harmonic functions are, by definition, the solutions to the sub-Laplacian equation. It was shown in \cite{CM} that in the case of the Heisenbgerg group,  the dimension threshold for such removable sets  is $\dim \H - 1 = 3$, where $\dim\H$ denotes the Hausdorff dimension of $\H$. See also \cite{CMT} for an extension of the previous result to all Carnot groups. As in the Euclidean case one has to handle a SIO whose kernel is the horizontal gradient of the fundamental solution of the sub-Laplacian. For example in $\H$, such a kernel can be explicitly written as
$$K(p):=\left( \frac{x (x^2+y^2)+yz}{((x^2+y^2)^2+z^2)^{3/2}}, \frac{y(x^2+y^2)-xz}{((x^2+y^2)^2+z^2)^{3/2}} \right)$$
for $p=(x,y,z) \in \H$. Currently we know very little about the action of this kernel on $3$-dimensional subsets of $\H$, nevertheless it has motivated research on SIOs on lower dimensional subets of $\H$, e.g. \cite{CM1} and the present paper, as well as the very recent study of quantitative rectifiability in $\H$, see \cite{CFO}.

\section{Preliminaries}
\label{sec:prelim}

Although we have already defined a metric on $\H$ we will also need the fact that there exists a natural path metric on $\H$. Notice that the Heisenberg group is a Lie group with respect to the group operation defined in \ref{eq:Heis}, and the Lie algebra of the left invariant vector fields in $\H$ is generated by the following vector fields
$$X:=\partial_x +y\partial_{z}, \  \ Y:=\partial_{y}-x \partial_{z}, \ \ T:=\partial_{z}.$$
The vector fields $X$ and $Y$ define the \emph{horizontal subbundle} $H \H$ of the tangent vector bundle of $\R^{3}$. For every point $p \in \H$ we will denote the horizontal fiber by $H_p \H$. Every such horizontal fiber is endowed with the left invariant scalar product $\langle \cdot,\cdot \rangle_p$ and the corresponding norm $|\cdot|_p$ that make the vector fields $X,Y,T$ orthonormal. 

\begin{df}
\label{horizontal}
An absolutely continuous curve $\gamma:[a,b]\ra \H$ will be called \emph{horizontal}, with respect to the vector fields $X,Y$ if 
$$\dot \gamma (t) \in H_{\gamma(t)}\H\ \ \text{for a.e.}\ t\in[a,b].$$
\end{df}

\begin{df}
\label{cc}
The \emph{Carnot-Carath\'eodory distance} of $p,q \in \H$ is
\begin{equation*}
d_{cc}(p,q)=\inf  \int_a^b |\dot \gamma (t)|_{\gamma(t)}dt
\end{equation*}
where the infimum is taken over all horizontal curves $\gamma:[a,b]\ra \H$
such that $\gamma(a)=p$ and $\gamma(b)=q.$
\end{df}

By Chow's theorem  the above set of curves joining $p$ and $q$ is not empty and hence $d_{cc}$ defines  a metric on $\H$. Furthermore the infimum in the definition can be replaced by a minimum. See \cite{blu} for more details.

\begin{remark} It follows by results of Pansu \cite{Pan1}, \cite{Pan2} that any $1$-AD regular curve is horizontal, hence the reader should keep in mind that our two main Theorems ( Theorems \ref{necesthm} and \ref{suffthm}) essentially involve horizontal curves.
\end{remark}

A point $p \in \H$ is called \emph{horizontal} if $p$ lies on the $xy$-plane. Observe that for horizontal points we can extend the dilations $\delta_r$ to all $\R$. Hence for a horizontal  point $p=(x,y,0)$ and $r \in \R$,
$$\delta_r(x,y,0)=(rx,ry,0).$$ 
We can now define an important family of curves in the Heisenberg group.

\begin{df}
Let $p,q \in \H$ such that $q$ is horizontal. The subsets of the form
$$\{p \cdot \delta_r(q): r \in \R\}$$ 
are called \emph{horizontal lines}.
\end{df}
Observe that horizontal lines  are horizontal curves with constant tangent vector. Note also that the horizontal lines going through a specified point in $\H$ span only two dimensions instead of three as in $\R^3$. This is a significant difference between Heisenberg and Euclidean geometry.

It is easy to see that the homomorphic projection $\pi : \H \to \R^2$ defined by
$$\pi(x,y,z)=(x,y),$$
is $1$-Lipshitz. We will also use the map $\tilde{\pi}:\H \ra \H$ defined by
$$\tilde{\pi}(x,y,z)=(x,y,0).$$ We stress that $\tilde{\pi}$ is \emph{not} a homomorphism.

\begin{df}[Horizontal interpolation]
For $p,q \in \H$,
$$\overline{p,q}=\{p \cdot \delta_r \tilde{\pi}(p^{-1} \cdot q): r \in [0,1]\}.$$
\end{df}

Note that $\overline{p,q}$ is a horizontal segment starting from $p$ traveling to the horizontal direction of $p^{-1}\cdot q$.

\begin{df}
\label{stdrd}Let $(X,d)$ be a metric space. We say that
$$k(\cdot,\cdot):X\times X\setminus\{x=y\}\rightarrow \R$$ is
an $n$-dimensional \emph{Calder\'{o}n-Zygmund (CZ)}-kernel if there exist
constants $c>0$ and $\eta$, with $0<\eta\leq1$, such that  for all $x,y\in X$, $x\neq y$:
\begin{enumerate}
\item $|k(x,y)| \leq \frac{c}{d(x,y)^n}$,
\item $|k(x,y)-k(x',y)|  + |k(y,x) -k(y,x')|\leq c \dfrac{d(x,x')^\eta}{d(x,y)^{n+\eta}}$ if $d(x,x')\leq d(x,y)/2$.
\end{enumerate}
\end{df}
For the next lemma, recall the definition \eqref{omega} of the functions $\Omega$.
\begin{lemma}
\label{stdrkernels}
Fix $m \in \N$, and let $k:\H \times \H\setminus\{x=y\}\rightarrow \R$ defined as
$$k(p,q)=\frac{\Omega(q^{-1}\cdot p)^m}{N(q^{-1}\cdot p)}.$$
Then $k$ is an $1$-dimensional CZ-kernel.
\end{lemma}

\begin{proof} We need to verify (1) and (2) from Definition \ref{stdrd}. Notice that (1) is immediate because by the definition of the Koran\'yi norm $NH(p) \leq N(p)$ for all $p \in \H$. For (2) we will use the fact that the function
$$f(p)=\frac{\Omega(p)^m}{N(p)}, \quad p \in \H \stm \{0\}$$
is $C^1$ away from the origin and it is also $-1$-homogeneous, that is
$$f(\delta_r(p))=\frac{1}{r} f(p)$$
for all $r>0$ and $p \in \H \stm \{0\}$. Hence by \cite[Proposition 1.7]{FolS}  there exists some constant $C>0$ such that for all $P,Q \in \H$ with $N(Q) \leq N(P)/2$
$$|f(P \cdot Q)-f(P)|\leq C \, \frac{N(Q)}{N(P)^2}.$$
Hence if $p,p',q \in \H$ such that $d_\H(p,p')\leq d_\H(p,q)/2$, 
\begin{equation}
\label{kerest}
\begin{split}
|k(p,q)-k(p',q)|&=|f(q^{-1} \cdot p)-f(q^{-1} \cdot p')|\\
&=|f(q^{-1} \cdot p)-f(q^{-1} \cdot p \cdot p^{-1} \cdot p')| \\
&\leq C \frac{N( {p'}^{-1} \cdot p)}{N(q^{-1} \cdot p)^2}=C\frac{d_\H(p',p)}{d_\H(p,q)^2}.
\end{split}
\end{equation}
Since $k$ is symmetric, from \eqref{kerest} we deduce that $k$ also satisfies (2) of Definition \ref{stdrd}. Hence the proof of the lemma is complete.
\end{proof}

In the sequel, we will use the notation $a \lesssim b$ or $a \gtrsim b$ to mean that there exists a universal constant $C$ so that $a \leq Cb$ or $a \geq Cb$.  This universal constant can change from instance to instance.  We let $a \asymp b$ mean both $a \lesssim b$ and $b \lesssim a$.  Given another fixed quantity $\alpha$, we let $a \lesssim_\alpha b$ and $b \lesssim_\alpha a$ mean that the quantity $C$ can depend only on $\alpha$.



\section{Necessary conditions}
In order to simplify notation, in the two remaining sections we will denote $d:=d_\H$,  $B(p,r):=B_\H(p,r)$ and $ab:=a \cdot b$ for $a,b \in \H$. 

Let $E \subset \H$ such that $\mu = \Hd^1|_E$ satisfy the 1-regularity condition
\begin{align*}
  \xi r \leq \mu(B(x,r)) \leq \xi^{-1} r, \qquad \forall x \in E, r > 0,
\end{align*}
for some $\xi < 1$.  We now recall the construction of \emph{David cubes}, which were  introduced by David in \cite{david-wavelets}.  David cubes can be constructed on any regular set of a geometrically doubling metric space. In particular for the set $E$, we obtain a constant $c >
0$ and a family of partitions $\Delta_{j}$ of $E$, $j \in \Z$, with the following properties;
\begin{itemize}
\item[(D1)] If $k \leq j$, $Q \in \Delta_{j}$ and $Q' \in
\Delta_{k}$, then either $Q \cap Q' = \emptyset$, or $Q \subset
Q'$. 
\item[(D2)] If $Q \in \Delta_{j}$, then $\diam Q \leq 2^{-j}$.
\item[(D3)] Every set $Q \in \Delta_{j}$ contains a set of the form
$B(p_{Q},c2^{-j}) \cap E$ for some $p_{Q} \in Q$.
\end{itemize}
The sets in $\Delta := \cup \Delta_{j}$ are called David cubes, or
\emph{dyadic cubes}, of $E$. Notice that $\diam(Q) \asymp 2^{-j}$ if $Q \in \Delta_j$. 
 For a cube $S \in \Delta$, we define
\begin{align*}
  \Delta(S) := \{Q \in \Delta : Q \subseteq S\}.
\end{align*}
Given a cube $Q \in \Delta$ and $\lambda \geq 1$, we define
\begin{align*}
  \lambda Q := \{x \in E : d(x,Q) \leq (\lambda - 1) \diam(Q)\}.
\end{align*}
It follows from (D1), (D2), and the $1$-regularity of $E$ that
$\mu(Q) \sim 2^{-j}$ for $Q \in \Delta_{j}$.

Define the positive symmetric -1-homogeneous kernel $K(p) = \Omega^8(p)/N(p) = \frac{NH(p)^8}{N(p)^9}$ .  For any $\ve > 0$, we can define the truncated operator as before
\begin{align*}
  T_1^\ve f(x) = \int_{d(y,x) > \ve} K(y^{-1}x) f(x) ~d\mu(y).
\end{align*}

\begin{proof}[Proof of Theorem \ref{necesthm}]
Our goal is to show that when $E$ lies on a rectifiable curve, there exists a uniform bound $C < \infty$ that can depend on $\xi$ so that
\begin{align}
  \|T_1^\ve \chi_S\|_{L^2(S)}^2 \leq C \mu(S), \qquad \forall S \in \Delta, \forall \ve > 0. \label{e:T_eps-bound}
\end{align}
We then apply the $T(1)$-theorem for homogeneous spaces, see e.g. \cite{denghan} and \cite{david-wavelets}, to deduce uniform $L^2$-boundedness of $T_1^\ve$ for all $\ve > 0$.  Note we may suppose $E$ is a 1-regular rectifiable curve as taking a subset can only decrease the $L^2$-bound of $T_1^\ve \chi_S$.

From now on we assume the 1-regular set $E$ actually lies on a rectifiable curve.
For $x \in E$ and $r > 0$, we define
\begin{align*}
  \beta_E(x,r) = \inf_L \sup_{z \in E \cap B(x,r)} \frac{d(z,L)}{r}.
\end{align*}
\begin{proposition} \label{p:tsp}
  There exists a constant $C \geq 1$ depending only $\xi$ so that for any $S \in \Delta$, we have
  \begin{align}
    \sum_{Q \in \Delta(S)} \beta(10Q)^4 \mu(Q) \leq C \mu(S). \label{e:tsp}
  \end{align}
\end{proposition}

\begin{proof}
    This essentially follows from Theorem I of \cite{li-schul-1} says that there exists some universal constant $C > 0$ so that
    \begin{align*}
        \int_\H \int_0^\infty \beta_E(B(x,t))^4 ~\frac{dt}{t^4} ~d\Hd^4(x) \leq C \Hd^1(E)
    \end{align*}
    when $E$ is simply a horizontal curve.  When $E$ is in addition 1-regular, it is a standard argument to use the Ahlfors regularity to lower bound this integral by a constant multiple---which can depend on $\xi$---of the left hand side of \eqref{e:tsp} (after intersecting $E$ with $S$).  In fact, one can easily show that the integral and sum are comparable up to multiplicative constants.

    One then gets
    \begin{align*}
        \sum_{Q \in \Delta(S)} \beta(10Q)^4 \mu(Q) \leq C \Hd^1(E \cap S) \lesssim_\xi \mu(S),
    \end{align*}
    where we again used 1-regularity of $E$ in the final inequality.
\end{proof}

We now fix $S \in \Delta$ a cube.  

Now define a positive even Lipschitz function $\psi : \R \to \R$ so that $\chi_{B(0,1/2)} \leq \psi \leq \chi_{B(0,2)}$.  We define
\begin{align*}
  \psi_j : \H &\to \R \\
  z &\mapsto \psi(2^j N(z)).
\end{align*}
We define $\phi_j := \psi_j - \psi_{j+1}$.  Thus, $\phi_j$ is supported on the annulus $B(0,2^{1-j}) \backslash B(0,2^{-2-j})$ in $\H$ and we have that
\begin{align}
  \chi_{\H \backslash B(0,2^{-n+1})} \leq \sum_{n \leq N} \phi_n \leq \chi_{\H \backslash B(0,2^{-n-2})}. \label{e:partial-phi-sum}
\end{align}
For each $j \in \Z$, we can define $K_{(j)} = \phi_j \cdot K$ and also
\begin{align*}
  T_{(j)}\chi_S(x) = \int_S K_{(j)}(y^{-1}x) ~d\mu(y).
\end{align*}
Define $S_N = \sum_{n \leq N} T_{(n)}$.  As the kernel $K$ is positive, we can easily get the following {\it pointwise} estimates for any positive function $f$ from \eqref{e:partial-phi-sum}
\begin{align*}
  0 \leq T_1^\ve f \leq S_{n+1} f, \qquad \forall \ve \geq 2^{-n}.
\end{align*}
Thus, to show uniform bound \eqref{e:T_eps-bound}, it suffices to show that there exists bound $C < \infty$ depending possibly on $\xi$ so that
\begin{align*}
  \|S_n \chi_S\|_{L^2(S)}^2 \leq C \mu(S), \qquad \forall S \in \Delta, \forall n \in \Z.
\end{align*}
We now fix $S \in \Delta_\ell$.

We will need the following lemma of \cite{li-schul-2}.

\begin{lemma}[Lemma 3.3 of \cite{li-schul-2}] \label{l:2pt-NH}
  For every $a,b \in \H$ and horizontal line $L \subset \H$, we have
  \begin{align}
    \max\{d(a,L),d(b,L)\} \geq \frac{1}{16} \frac{NH(a^{-1}b)^2}{d(a,b)}. \label{e:2pt-NH}
  \end{align}
\end{lemma}

\begin{lemma} \label{l:Tj-beta}
  For any $j \in \Z$ and $x \in E$, we have
  \begin{align}
    T_{(j)}1(x) \lesssim_\xi \beta_E(x,2^{1-j})^4. \label{e:Tj-beta}
  \end{align}
\end{lemma}

\begin{proof}
  Define the annulus $A = E \cap A(x,2^{-2-j},2^{1-j})$.  Then
  \begin{multline*}
    T_{(j)}1(x) \leq \int_E \phi_j(y^{-1}x) K(y^{-1}x) ~d\mu(y) \leq 2^{j+2} \int_A \frac{NH(y^{-1}x)^8}{N(y^{-1}x)^8} ~d\mu(y) \\
    \lesssim_\xi \sup_{y \in A} \frac{NH(y^{-1}x)^8}{d(x,y)^8}.
  \end{multline*}
  It suffices to show $\frac{NH(y^{-1}x)^8}{d(x,y)^8} \leq 8^4 \beta_E(B(x,2^{1-j}))^4$ when $y \in A$.  This follows easily from \eqref{e:2pt-NH}.  Indeed, as $y \in A$, we have that $d(x,y) \geq 2^{-j-2}$.  We can then find a horizontal line so that
  \begin{align*}
    \beta_{\{x,y\}}(B(x,2^{1-j})) = \frac{\max\{d(x,L),d(y,L)\}}{2^{1-j}} \geq \frac{\max\{d(x,L),d(y,L)\}}{8d(x,y)} \overset{\eqref{e:2pt-NH}}{\geq} \frac{NH(x^{-1}y)^2}{128d(x,y)^2}.
  \end{align*}
  The statement now follows as $\beta_E(B(x,2^{1-j})) \geq \beta_{\{x,y\}}(B(x,2^{1-j}))$.
\end{proof}

We now have the following easy corollary.
\begin{corollary}
  Let $R \in \Delta_j$.  Then for any $\alpha > 0$, we have
  \begin{align}
    \int_R T_{(j)}1(x)^\alpha ~d\mu(x) \lesssim_\xi \beta_E(10R)^{4\alpha} \mu(R). \label{e:Tj-int-beta}
  \end{align}
\end{corollary}

\begin{remark}
  Note that we may replace the constant 1 function in \eqref{e:Tj-beta} and \eqref{e:Tj-int-beta} with any positive function $f \leq 1$ (such as $f = \chi_S$ for some $S \in \Delta$).  This is again because the kernel of $T_j$ is positive and so respects the partial ordering of positive functions.
\end{remark}

For any $Q \in \Delta$, we can also define
\begin{align*}
  T_Q \chi_S := \chi_Q T_{(j(Q))} \chi_S.
\end{align*}
Thus, we have that
\begin{align*}
  S_n\chi_S = \sum_{j \leq n} T_{(j)} \chi_S = \sum_{j \leq n} \sum_{Q \in \Delta_j} T_Q \chi_S.
\end{align*}
and so
\begin{align}
  \|S_n \chi_S\|_{L^2(S)}^2 = \sum_{j \leq n} \|T_{(j)} \chi_S\|_{L^2(S)}^2 + 2 \sum_{j < k \leq n} \langle T_{(j)}\chi_S, T_{(k)}\chi_S \rangle \label{e:q-ortho}
\end{align}
where the inner product $\langle \cdot,\cdot \rangle$ is integration on $S$.  We will bound the two terms on the right hand side separately.

Let $S^* \in \Delta_{\ell-2}$ be such that $S \subset S^*$.  It follows from (D1) that $S^*$ is unique for $S$.  It follows from the $\phi_j$ factor and the fact that cubes of $\Delta_\ell$ have diameter at most $2^{-\ell}$ that $T_{(j)}\chi_S(x) = 0$ for $x \in S \in \Delta_\ell$ whenever $j < \ell - 2$.  Thus, as $S \in \Delta_\ell$, we have
\begin{align}
  \sum_{j \leq n} \|T_{(j)} \chi_S\|_{L^2(S)}^2 \leq \sum_{\ell - 2 \leq j \leq n} \sum_{Q \in \Delta_j, Q \subseteq S} \int_Q T_{(j)}\chi_S(x)^2 ~d\mu(x) \overset{\eqref{e:Tj-int-beta}}{\lesssim_\xi} \sum_{Q \in \Delta(S^*)} \beta(10Q)^8 \mu(Q). \label{e:diag}
\end{align}

We now have to bound the off diagonal terms of \eqref{e:q-ortho}.  We have that
\begin{align}
  \sum_{j \geq \ell - 2} \sum_{j < k \leq n} \int_S T_{(j)} \chi_S(x) \cdot T_{(k)}(x) \chi_S ~d\mu(x) &\overset{\eqref{e:Tj-beta}}{\lesssim_\xi} \sum_{j \geq \ell - 2} \sum_{Q \in \Delta_j(S)} \beta(10Q)^4 \sum_{k > j} \int_Q T_{(k)} \chi_S ~d\mu(x) \notag \\
  &\overset{\eqref{e:Tj-int-beta}}{\lesssim_\xi} \sum_{Q \in \Delta(S^*)} \beta(10Q)^4 \sum_{R \in \Delta(Q)} \beta(10R)^4 \mu(R) \notag \\
  &\overset{\eqref{e:tsp}}{\lesssim_\xi} C \sum_{Q \in \Delta(S^*)} \beta(10Q)^4 \mu(Q). \label{e:off-diag}
\end{align}
Note that the constants hidden in the $\lesssim$ of \eqref{e:diag} and \eqref{e:off-diag} do not depend on $S$ or $n$.

Altogether, we have
\begin{align*}
  \|S_n\chi_S\|_{L^2(S)}^2 \overset{\eqref{e:q-ortho} \wedge \eqref{e:diag} \wedge \eqref{e:off-diag}} {\lesssim_\xi} \sum_{Q \in \Delta(S^*)} \beta(10Q)^4 \mu(Q) \overset{\eqref{e:tsp}}{\lesssim_\xi} \mu(S^*) \lesssim_{\xi,c} \mu(S),
\end{align*}
where we used properties (D2), (D3), and 1-AD regularity of $E$ in the last inequality.
\end{proof}

We now demonstrate how using a positive kernel leads to an easy proof of Corollary \ref{pvcor}.

\begin{proof}[Proof of Corollary \ref{pvcor}]
    First suppose that $f \in L^2(E)$ is a nonnegative function.  Then as the kernel $K_1$ is positive, we have for fixed $p \in E$ that $T_1^\ve f(p)$ is a monotonically increasing sequence as $\ve \to 0$ and so $\mathrm{p.v.} T_1f(p) := \lim_{\ve \to 0} T_1^\ve f(p)$ is a well defined function, although it be infinity.  By Theorem \ref{necesthm}, we get that there exists some $C > 0$ so that
    \begin{align*}
      \sup_{\ve > 0} \int (T_1^\ve f)^2 ~d\mu \leq C \int f^2 ~d\mu.
    \end{align*}
    Thus, by Fatou's lemma, we get
    \begin{align*}
      \int (\mathrm{p.v.}T_1f)^2 d\mu \leq \liminf_{\ve \to 0} \int (T_1^\ve f)^2 \leq C \int f^2 d\mu.
    \end{align*}
    This then proves the corollary for nonnegative functions.

    Now let $f \in L^2(E)$ be a real valued function.  We can decompose $f = f^+ - f^-$ where $f^+ = \max \{f,0\}$ and $f^- = \max\{-f,0\}$.  Then $\max(\|f^+\|_{L^2(E)}, \|f^-\|_{L^2(E)}) \leq \|f\|_{L^2(E)}$ and so we get that the principal values of $f^+$ and $f^-$ under $T_1$ are controlled by $C\|f\|_{L^2(E)}$.  Thus, the principal values have to be finite almost everywhere and so we get $\mathrm{p.v.} T_1f = \mathrm{p.v.} T_1f^+ - \mathrm{p.v.} T_1f^-$ as $L^2(E)$ functions.  Additionally, we get
    \begin{align*}
      \|\mathrm{p.v.} T_1f\|_{L^2(E)} \leq \|\mathrm{p.v.} T_1f^+\|_{L^2(E)} + \|\mathrm{p.v.} T_1f^-\|_{L^2(E)} \leq 2C\|f\|_{L^2(E)}.
    \end{align*}
    This proves the entire corollary.
\end{proof}

\section{Sufficient conditions}

We will need the following ``triangle inequality'' for this section.

\begin{lemma}[$NH^2$ triangle inequality]
    Let $a,b,c \in \H$ and let $A$ be the (unsigned) area of the triangle in $\R^2$ with vertices $\pi(a), \pi(b), \pi(c)$.  For the following four quantities
    \begin{align*}
        A, \quad NH(a^{-1}b)^2, \quad NH(b^{-1}c)^2, \quad NH(c^{-1}a)^2,
    \end{align*}
    any one of these numbers is less than the sum of the other three.
\end{lemma}

\begin{proof}
    Let us first show $A$ is less than the sum of the $NH^2$.  Since everything is invariant under left-translation, we may suppose $c = (0,0,0)$, $a = (x,y,t)$, and $b = (x',y',t')$.  Then $NH(c^{-1}a)^2 = |t|$ and $NH(b^{-1}c)^2 = |t'|$ and we have
    \begin{multline*}
      A = \frac{|x'y - xy'|}{2} \leq \left| \frac{x'y - xy'}{2} - t + t' \right| + |t'| + |t| \\
      \leq NH(a^{-1}b)^2 + NH(b^{-1}c)^2 + NH(c^{-1}a)^2.
    \end{multline*}

    We now show that $NH(a^{-1}b)^2$ is less than the sum of the other three quantities.  We will keep the same normalization as the last case.
    \begin{multline*}
      NH(a^{-1}b)^2 = \left| \frac{x'y - xy'}{2} - t + t' \right| \leq \frac{|x'y - xy'|}{2} + |t'| + |t| \\
      \leq A + NH(b^{-1}c)^2 + NH(c^{-1}a)^2.
    \end{multline*}
\end{proof}

For $r < R$ and $x \in \H$, we can define the annulus
\begin{align*}
  A(x,r,R) := \{y \in \H : d(x,y) \in (r,R)\}.
\end{align*}
For three points $p_1,p_2,p_3$ in a $\H$, we define
\begin{align*}
  \partial(p_1,p_2,p_3) = \min_{\sigma \in S_3} (d(p_{\sigma(1)},p_{\sigma(2)}) + d(p_{\sigma(2)},p_{\sigma(3)}) - d(p_{\sigma(1)},p_{\sigma(3)})).
\end{align*}
For $\alpha \in (0,1)$, $r > 0$, and a metric space $X$, we let $\Sigma_X(\alpha,r)$ denote the triples of points $(p_1,p_2,p_3) \in X$ so that
\begin{align*}
  \alpha r \leq d(p_i,p_j) \leq r, \qquad \forall i \neq j.
\end{align*}
We also let $\Sigma_X(\alpha) = \bigcup_{r > 0} \Sigma_X(\alpha,r)$.  For notational convenience, we will drop the $X$ subscript when we want $X = E$ where $E$ is the 1-regular set of the hypothesis of Theorem \ref{suffthm}.

\begin{lemma} \label{l:NH-strip}
  Let $(p_1,p_2,p_3) \in \Sigma(\alpha,r)$.  If for some $\ve \in (0,1/2)$ we have
  \begin{align}
    NH(p_i^{-1}p_j) \leq \ve d(p_i,p_j), \label{e:small-NH}
  \end{align}
  then the point $\pi(p_i) \in \R^2$ is contained in the strip around the line $\overline{\pi(p_{i+1}),\pi(p_{i+2})}$ of width $16 \alpha^{-1} \ve^2 r$.
\end{lemma}

\begin{proof}
  We will view $\pi(p_2),\pi(p_3)$ as the base of a triangle with top vertex $\pi(p_1)$.  It suffices to bound the height.  We let $A$ denote the area of the triangle.

  Suppose $A \geq 4 \ve^2 r^2$.  We have by the $NH^2$ triangle inequality that
  \begin{align*}
    NH(p_2^{-1}p_3)^2 \geq A - NH(p_1^{-1}p_2)^2 - NH(p_1^{-1}p_3)^2 \overset{\eqref{e:small-NH}}{\geq} 2 \ve^2 r^2.
  \end{align*}
  This is a contradiction of \eqref{e:small-NH}.

  Thus, we may assume $A \leq 4\ve^2 r^2$.  But if $NH(p_2^{-1}p_3) \leq d(p_2,p_3)/2$, then $|\pi(p_2) - \pi(p_3)| \geq d(p_2,p_3)/2 \geq \alpha r/2$.  Thus, the height of the triangle is less than
  \begin{align*}
    \frac{2A}{|\pi(p_2) - \pi(p_3)|} \leq \frac{16}{\alpha} \ve^2 r.
  \end{align*}
\end{proof}

Given $u,v,w \in \H$, we denote the largest and second largest quantity of
$$\left\{\frac{NH(u^{-1}v)}{d(u,v)},\frac{NH(v^{-1}w)}{d(v,w)},\frac{NH(u^{-1}w)}{d(u,w)} \right\}$$
by $\gamma_1(u,v,w)$ and $\gamma_2(u,v,w)$, respectively.

\begin{lemma} \label{l:partial-NH1}
  For all $\alpha > 0$, there exists a constant $c_1 > 0$ so that if $(p_1,p_2,p_3) \in \Sigma(\alpha,r)$, then
  \begin{align*}
    \partial(p_1,p_2,p_3) \leq c_1 \gamma_1(p_1,p_2,p_3)^4 r.
  \end{align*}
\end{lemma}

\begin{proof}
  Let $\gamma = \gamma_1(p_1,p_2,p_3)$, and we may suppose without loss of generality that
  \begin{align*}
    \partial(p_1,p_2,p_3) = d(p_1,p_2) + d(p_2,p_3) - d(p_1,p_3).
  \end{align*}

  Suppose first that $\gamma < c$ for some $c > 0$ to be determined soon.  Then
  \begin{align}
    NH(p_i^{-1}p_j) \leq \gamma d(p_i, p_j) < c d(p_i,p_j), \qquad \forall i \neq j, \label{e:small-all-NH}
  \end{align}
  and so
  \begin{equation*}
  \begin{split}
    |\pi(p_i) - \pi(p_j)| &= \left( d(p_i,p_j)^4 - NH(p_i^{-1}p_j)^4 \right)^{1/4} \geq \left( 1-  c^4 \right)^{1/4} d(p_i,p_j).
    \end{split}
  \end{equation*}
  By taking $c$ small enough, we get, that $(\pi(p_1),\pi(p_2),\pi(p_3)) \in \Sigma_{\R^2}(\alpha/2)$ and by Taylor expansion of the Kor\'anyi norm, that 
  \begin{align*}
    d(p_i,p_j) \leq |\pi(p_i) - \pi(p_j)| + \frac{NH(p_i^{-1}p_j)^4}{|\pi(p_i) - \pi(p_j)|^3} \leq |\pi(p_i) - \pi(p_j)| + (1 - c^4)^{-3/4}  \gamma^4 r,
  \end{align*}
  and so
  \begin{align}
    \partial(p_1,p_2,p_3) \leq |\pi(p_1) - \pi(p_2)| + |\pi(p_2) - \pi(p_3)| - |\pi(p_1) - \pi(p_3)| + 2 (1 - c^4)^{-3/4} \gamma^4r. \label{e:partial-bnd}
  \end{align}
  As $(\pi(p_1),\pi(p_2),\pi(p_3)) \in \Sigma_{\R^2}(\alpha/2)$, we get by a Taylor approximation of the Euclidean metric that
  \begin{align}
    |\pi(p_1) - \pi(p_2)| + |\pi(p_2) - \pi(p_3)| - |\pi(p_1) - \pi(p_3)| \lesssim_\alpha \frac{h^2}{r}, \label{e:R2-partial-bnd}
  \end{align}
  where $h$ be the height of the triangle in $\R^2$ defined by $\pi(p_i)$ with base $\pi(p_1),\pi(p_3)$.  From \eqref{e:small-NH} and \eqref{e:small-all-NH}, we have
  \begin{align}
    h \leq 16 \alpha^{-1} \gamma^2 r. \label{e:ht-bnd}
  \end{align}
  The result now follows from \eqref{e:partial-bnd}, \eqref{e:R2-partial-bnd}, and \eqref{e:ht-bnd}.

  Now suppose $\gamma \geq c$.  As $\partial(p_1,p_2,p_3) \leq 3r$, the lemma trivially follows.
\end{proof}

We let $E \subset \H$ be a set with $\mu = \Hd^1|_E$ satisfying the following estimate
\begin{align*}
    \xi r \leq \mu(B(x,r)) \leq \xi^{-1} r, \qquad \forall x \in E, r > 0,
\end{align*}
where $\xi \leq 1$.

\begin{lemma} \label{l:trichotomy}
  Let $E \subset \H$ be a 1-regular set and $\alpha \in (0,1)$.  There exists $c_2 \geq 1$ depending on $\alpha$ and $\xi$ so that if $(p_1,p_2,p_3) \in \Sigma(\alpha,r)$, then
  one of the following is true:
  \begin{enumerate}
    \item $\gamma_1(p_1,p_2,p_3) \leq c_2 \gamma_2(p_1,p_2,p_3)$,
    \item after a possible reindexing of $p_i$, there exists a set $V \subseteq E \cap B(p_1, \alpha r/10)$ with $\mu(V) \geq r/c_2$ so that for every $x \in V$ we have
    \begin{align*}
      \gamma_1(p_1,p_2,p_3) \leq c_2 \gamma_2(x,p_2,p_3),
    \end{align*}
    and $(x,p_2,p_3) \in \Sigma(c_2^{-1})$.
    \item after a possible reindexing of $p_i$, there exists sets $W_1, W_2 \subseteq E \cap B(p_1, \alpha r/5)$ with $\mu(W_1), \mu(W_2) \geq r/c_2$ so that for all $(x,y) \in W_1 \times W_2$, we have
    \begin{align*}
      \gamma_1(p_1,p_2,p_3) \leq c_2 \gamma_2(p_1,x,y),
    \end{align*}
    and $(p_1,x,y) \in \Sigma(c_2^{-1}, r)$.
  \end{enumerate}
\end{lemma}

\begin{proof}
  Throughout this proof, we will give a finite series of lower bounds for $c_2$.  The final $c_2$ will then just be the maximum of these lower bounds.  For simplicity of notation, let $\gamma_i = \gamma_i(p_1,p_2,p_3)$.  We may of course suppose that $\gamma_2 \leq c \gamma_1$ for some small $c > 0$ depending on $\alpha$ and $\xi$ to be determined as otherwise Condition (1) would be satisfied.  Without loss of generality, we can assume that $\gamma_1 = NH(p_2^{-1}p_3)/d(p_2,p_3)$.  Let $A$ denote the area of the triangle in $\R^2$ with vertices $\pi(p_i)$.  Then we have from $NH^2$ triangle inequality that
  \begin{align*}
    NH(p_2^{-1}p_3)^2 \leq NH(p_1^{-1}p_2)^2 + NH(p_1^{-1}p_3)^2 + A,
  \end{align*}
  and so if we set $c < \alpha/2$ (while still allowing ourselves to take $c$ smaller) then
  \begin{align}
    A \geq \frac{\alpha^2}{2} \gamma_1^2 r^2. \label{e:big-triangle}
  \end{align}

  Fix $\lambda \in (0,1)$ depending only $\xi$ so that
  \begin{align*}
    \mu(A(x,\lambda \ell,\ell)) \geq \frac{1}{2} \xi \ell, \qquad \forall x \in E, \ell > 0.
  \end{align*}
  Suppose now $A(p_1, \lambda \alpha r/10, \alpha r/10)$ contains a subset $S$ of $\mu$-measure at least $\xi \alpha r/40$ so that
  \begin{align}
    \frac{NH(x^{-1}p_1)}{d(x,p_1)} < c \gamma_1, \qquad \forall x \in S. \label{e:small-NH-p1}
  \end{align}
  If there is a further subset $V \subseteq S$ with $\mu(V) \geq \xi \alpha r/80$ so that $NH(x^{-1}p_2) \geq c \gamma_1 d(x,p_2)$ for each $x \in V$, then we are done as we've satisfied Condition (2) for large enough $c_2$ if we keep $p_2,p_3$ and draw $x$ from $V$.
  
  Thus, suppose there is a subset $V \subseteq S$ with $\mu(V) \geq \xi \alpha r/80$ and 
  \begin{align}
    \frac{NH(x^{-1}p_2)}{d(x,p_2)} < c \gamma_1, \qquad \forall x \in V. \label{e:small-NH-p2}
  \end{align}
  Recalling
  \begin{align}
    d(x,p_1) \in \left[\frac{\lambda \alpha r}{10}, \frac{\alpha r}{10} \right], \quad d(x,p_2) \in \left[ \frac{r}{2}, 2r \right], \quad \forall x \in V \subseteq A\left(p_1, \frac{\lambda \alpha r}{10}, \frac{\alpha r}{10} \right) \label{e:x-comp}
  \end{align}
  we get from \eqref{e:small-NH-p1}, \eqref{e:small-NH-p2}, and Lemma \ref{l:NH-strip} that for every $x \in V$, $\pi(x)$ lies in the strip around $\overline{\pi(p_1),\pi(p_2)}$ of width
  \begin{align}
    w = \frac{640}{\lambda \alpha} c^2 \gamma_1^2 r. 
    \label{e:small-width}
  \end{align}
  As $NH(x^{-1}p_1) < c\gamma_1 d(x,p_1)$, we easily get (supposing $c$ is small enough) that
  \begin{align}
    |\pi(x) - \pi(p_1)| \geq \frac{1}{2} d(x,p_1) \overset{\eqref{e:x-comp}}{\geq} \frac{\lambda \alpha}{20} r. \label{e:pi-x-p_1}
  \end{align}
  As $d(p_1,p_2) \leq r$, we get that the height of the triangle given by $\pi(p_i)$ with base $\overline{\pi(p_1),\pi(p_2)}$ is then at least
  \begin{align*}
    h \geq \frac{2A}{d(p_1,p_2)} \overset{\eqref{e:big-triangle}}{\geq} \alpha^2 \gamma_1^2 r.
  \end{align*}

\begin{figure}
\centering
\includegraphics[scale = 0.55]{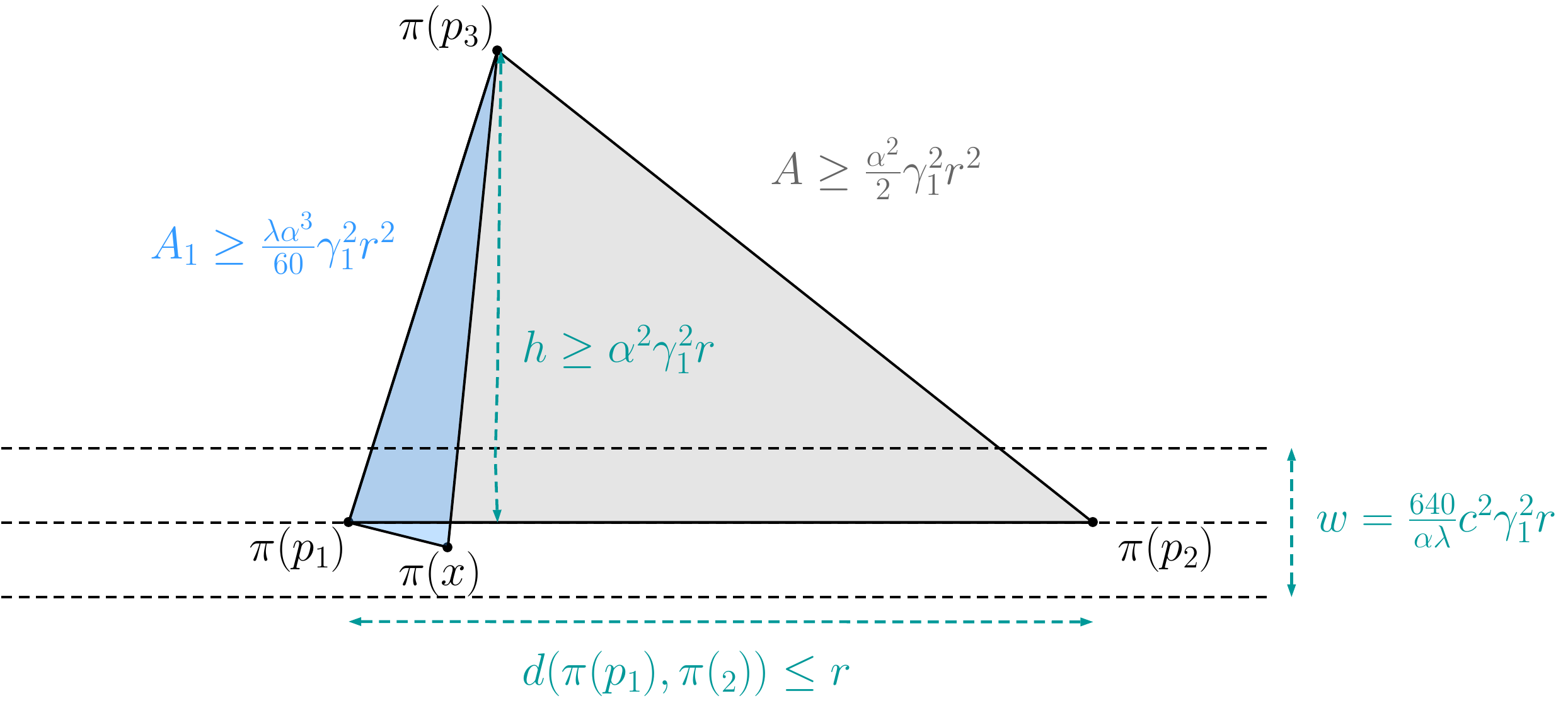}
\caption{$A$ denotes the area of the triangle determined by $\pi(p_i),i=1,2,3,$ and $A_1$ denotes the area of the triangle determined by $\pi(p_1), \pi(p_3)$ and $\pi(x)$.}
\label{fig1}
\end{figure}
  Let $A_1$ denote the area of the triangle determined by $\pi(p_1),\pi(x),\pi(p_3)$.  By \eqref{e:small-width}, we have that $w$ is at most some constant multiple (depending on $\alpha$ and $\lambda$) of $c^2 h$.  Thus, if we choose $c$ small enough to get $\pi(x)$ sufficiently close to the line $\overline{\pi(p_1),\pi(p_2)}$ compared to $h$, we get
  \begin{align*}
    A_1 \geq \frac{h |\pi(p_1) - \pi(x)|}{3} \overset{\eqref{e:pi-x-p_1}}{\geq} \frac{\lambda \alpha^3}{60} \gamma_1^2 r^2.
  \end{align*}
  See Figure \ref{fig1} for an illustration of these triangles.

  Now using the $NH^2$ triangle inequality, we get
  \begin{equation*}
  \begin{split}
    \frac{\alpha^3 \lambda}{60} \gamma_1^2 r^2 &\leq A_1 \leq NH(x^{-1}p_1)^2 + NH(p_1^{-1}p_3)^2 + NH(x^{-1}p_3)^2 \\
    & \overset{\eqref{e:small-NH-p1} \wedge \eqref{e:x-comp}}{\leq} 2c^2 \gamma_1^2 r^2 + NH(x^{-1}p_3)^2.
  \end{split}
  \end{equation*}
  Thus, if we choose $c$ small enough compared to $\alpha$ and $\lambda$ once and for all, we get that
  \begin{align*}
    NH(x^{-1}p_3) \geq \frac{\sqrt{\alpha^3 \lambda}}{10} \gamma_1 r \geq \frac{\sqrt{\alpha^3 \lambda}}{20} \gamma_1 d(x,p_3).
  \end{align*}
  We now see that we can satisfy Condition (2) for sufficiently large $c_2$ by keeping $p_2,p_3$ and drawing $x$ from $V$

  Thus, we may suppose that $E \cap A(p_1,\lambda \alpha r/10,\alpha r/10)$ contains a subset $S$ so that $\mu(S) \geq \xi\alpha r/40$ and
  $$NH(z^{-1}p_1) \geq c \gamma_1 d(z,p_1), \qquad \forall z \in S.$$
  Using 1-regularity of $E$, an elementary, although tedious, packing argument shows that there exist $\eta, \tau < \lambda \alpha / 100$ depending only on $\alpha$ and $\xi$ and points $x',y' \in E \cap A(p_1,\lambda \alpha r/10, \alpha r/10)$ so that $d(x',y') \geq 10 \tau r$ and
  $$\min\{ \mu(S \cap B(x',\tau r)), \mu(S \cap B(y',\tau r)) \} \geq \eta r.$$
  Note by the triangle inequality that we get
  \begin{align*}
    B(x', \tau r), B(y', \tau r) \subseteq A(p_1, \lambda \alpha r/20, \alpha r/5)  
  \end{align*}
  Thus, after setting $c_2$ large enough, we've satisfied Condition (3) with $W_1 = S \cap B(x',\tau r)$ and $W_2 = S \cap B(y', \tau r)$, which would completely finish the proof of the lemma.  We will present a quick sketch of the packing argument and leave the details to the reader.
  
  Find a maximal $\tau r$-separated net $\mathcal{N}$ of $E \cap B(p_1, \alpha r)$ for $\tau > 0$ to be determined.  By 1-regularity, we have $\# \mathcal{N} \gtrsim \alpha/\tau$.  First use 1-regularity of $E$ to find $M \geq 1$ so that any subset $S \subseteq \mathcal{N}$ for which $\# S \geq M$ must contain $x',y' \in S$ so that $d(x',y') \geq 10 \tau r$.  Now $\{B(x,\tau r) : x \in \mathcal{N}\}$ is a covering of $B(p_1, \alpha r/10)$.  Define $\mathcal{B} = \{B(x,\tau r) : x \in \mathcal{N}, \mu(S \cap B(x,r)) \geq \eta r\}$.  By choosing $\eta$ small enough relative to $\alpha \tau$, we can use 1-regularity of $E$ and the fact that $\mu(S) \gtrsim \alpha r$ to get that $\# \mathcal{B} \gtrsim \alpha \mathcal{N} \gtrsim \alpha^2/\tau$ (with no dependence on $\eta$).  Now simply choose $\tau$ small enough so that $\# \mathcal{B} \geq M$.  One then finds two ball $B(x',\tau r), B(y',\tau r) \in \mathcal{B}$ so that $d(x',y') \geq 10 \tau r$, which finishes the sketch.
\end{proof}

For $x,y \in E$, we let
\begin{align*}
  \Sigma(\alpha,r;x) &:= \{(y,z) \in E^2 : (x,y,z) \in \Sigma(\alpha,r)\}, \\
  \Sigma(\alpha;x,y) &:= \{z \in E : (x,y,z) \in \Sigma(\alpha)\}.
\end{align*}
One easily has that there exists some constant $c_3 \geq 1$ depending on $\xi$ so that
\begin{align*}
  \frac{1}{c_3} r^2 \leq \mu \times \mu(\Sigma(\alpha,r;x)) \leq c_3 r^2, \quad \frac{1}{c_3} d(x,y) \leq \mu(\Sigma(\alpha;x,y)) \leq c_3 d(x,y).
\end{align*}

For simplicity of notation, we will adopt the convention that the integral $\int_A f(x) ~dx$ means $\int_A f(x) ~d\mu(x)$ when $A \subseteq E$.  Recall that for three points $p_1,p_2,p_3$ in a metric space $X$, its Menger curvature $c(p_1,p_2,p_3) \in \R$ is defined as
\begin{align*}
  c(p_1,p_2,p_3) = \frac{1}{R},
\end{align*}
where $R$ is the radius of the circle in $\R^2$ passing through a triangle defined by the vertices $p_1',p_2',p_3' \in \R^2$ where $d(p_i,p_j) = |p_i' - p_j'|$.

\begin{proposition}
  For any $\alpha > 0$, there exist $c_4 \geq 1$ so that
  \begin{align}
    \iiint_{\Sigma(\alpha)} c(x,y,z)^2 ~dx ~dy ~dz \leq c_4 \iiint_{\Sigma(c_4^{-1})} \frac{\gamma_1(x,y,z)^2 \gamma_2(x,y,z)^2}{\diam(\{x,y,z\})^2} ~dx ~dy ~dz. \label{e:gamma-bound}
  \end{align}
\end{proposition}

\begin{proof}
  We have by \cite{hahlomaa-2} that there exists some $\tau > 0$ depending on $\alpha$ so that if $(x,y,z) \in \Sigma(\alpha)$, then
  \begin{align}
    c(x,y,z)^2 \leq \tau \diam(\{x,y,z\})^{-3} \partial(x,y,z). \label{e:c2-partial}
  \end{align}
  By Lemma \ref{l:partial-NH1}, we have that there exist $c_1 > 0$ so that
  \begin{align}
    \iiint_{\Sigma(\alpha)} \diam(\{x,y,z\})^{-3} \partial(x,y,z) ~dx ~dy ~dz \leq c_1 \iiint_{\Sigma(\alpha)} \frac{\gamma_1(x,y,z)^4}{\diam(\{x,y,z\})^2} ~dx ~dy ~dz. \label{e:partial-gamma}
  \end{align}
  We now decompose $\Sigma(\alpha)$ into three pieces.  For $i = 1,2,3$, let $S_i \subseteq \Sigma(\alpha)$ denote the triples of points for which Condition (i) of Lemma \ref{l:trichotomy} holds for some $r > 0$ (that can depend on the triple of points).  Note $\Sigma(\alpha) \subseteq S_1 \cup S_2 \cup S_3$, but this decomposition need not be disjoint.

  It will be convenient to define the functions
  \begin{align*}
    f(x,y,z) := \frac{\gamma_1(x,y,z)^4}{\diam(\{x,y,z\})^2}, \quad g(x,y,z) := \frac{\gamma_1(x,y,z)^2 \gamma_2(x,y,z)^2}{\diam(\{x,y,z\})^2}.
  \end{align*}
  We trivially have that
  \begin{align}
    \iiint_{S_1} f(x,y,z) ~dx ~dy ~dz \leq c_2^2 \iiint_{S_1} g(x,y,z) ~dx ~dy ~dz. \label{e:S1-leq}
  \end{align}

  When we write a triple of points $(x,y,z) \in S_2$, we will always assume $y,z$ play the role of $p_2,p_3$ in Condition (2).  Now let $(x,y,z) \in S_2 \cap \Sigma(\alpha)$.  We then have that there exists a subset 
  with $\mu(V) \geq r/c_2$,
  \begin{align*}
    f(x,y,z) \leq c_2 g(u,y,z), \qquad \forall u \in V.
  \end{align*}
  We then have that
  \begin{align*}
    f(x,y,z) \leq c_2 \frac{1}{\mu(V)} \int_V g(u,y,z) ~du.
  \end{align*}
  We also have that $(u,y,z) \in \Sigma(c_2^{-1})$ for all $ u \in V$ and so
  \begin{align*}
    \int_{\Sigma(\alpha;y,z)} f(x,y,z) ~dx &\leq c_2 \frac{\mu(\Sigma(\alpha;y,z))}{\mu(V)} \int_V g(u,y,z) ~du \\
    &\leq c_2^2 c_3 \int_{\Sigma(c_2^{-1};y,z)} g(u,y,z) ~du.
  \end{align*}
  Now we have
  \begin{align}
    \iiint_{S_2} f(x,y,z) ~dx ~dy ~dz &= \iiint_{\Sigma(\alpha)} {\bf 1}_{S_2} f(x,y,z) ~dx ~dy ~dz \notag \\
    &\leq \int_E \int_E \int_{\Sigma(\alpha;y,z)} {\bf 1}_{S_2} f(x,y,z) ~dx ~dy ~dz \notag \\
    &\leq c_2^2 c_3 \int_E \int_E \int_{\Sigma(c_2^{-1};y,z)} g(x,y,z) ~dx ~dy ~dz \notag \\
    &\leq 6c_2^2 c_3 \iiint_{\Sigma(c_2^{-1})} g(x,y,z) ~dx ~dy ~dz. \label{e:S2-leq}
  \end{align}

  For $S_3$, we will write the points $(x,y,z)$ with the understanding that $z$ plays the role of $p_1$ in Condition (3).  Now let $(x,y,z) \in S_3 \cap \Sigma(\alpha/2,r)$.  In a way similar to above, we can use the properties of the conclusion of Property (3) to get that
  \begin{align*}
    f(x,y,z) \leq c_2^2 c_3 \iint_{\Sigma(c_2^{-1},r;z)} g(u,v,z) ~du ~dv.
  \end{align*}
  
  It is elementary to see that if $(x,y,z) \in \Sigma(\alpha)$, then $\int_0^\infty {\bf 1}_{\{r : (x,y) \in \Sigma(\alpha/2,r;z)\}} \frac{dr}{r} \asymp_\alpha 1$.  Here, we need the extra factor of 1/2 in case $(x,y,z)$ acheives tightness in the $\Sigma(\alpha)$ condition.  We can now decompose the integral
  \begin{align}
    \iiint_{S_3} f(x,y,z) ~dx ~dy ~dz &\lesssim_\alpha \iiint_{S_3} f(x,y,z) \int_0^\infty {\bf 1}_{\{r : (x,y) \in \Sigma(\alpha/2,r;z)\}} \frac{dr}{r} ~dx ~dy ~dz \notag \\
    &\leq \int_E \int_0^\infty \iint_{\{(x,y) \in \Sigma(\alpha/2,r;z) : (x,y,z) \in S_3\}} f(x,y,z) ~dx ~dy \frac{dr}{r} ~dz \notag \\
    &\leq c_2^2 c_3 \int_E \int_0^\infty \iint_{\Sigma(c_2^{-1},r;z)} g(u,v,z) ~du ~dv \frac{dr}{r} ~dz \notag \\
    &\lesssim_\alpha \int_{\Sigma(c_2^{-1})} g(x,y,z) \int_0^\infty {\bf 1}_{\{r : (u,v) \in \Sigma(c_2^{-1},r;z)\}} \frac{dr}{r} ~du ~dv ~dz \\
    &\lesssim \iiint_{\Sigma(c_2^{-1})} g(x,y,z) ~dx ~dy ~dz. \label{e:S3-leq}
  \end{align}
  In the second and penultimate inequality, we used Fubini.  We then get the conclusion from \eqref{e:c2-partial}, \eqref{e:partial-gamma}, \eqref{e:S1-leq}, \eqref{e:S2-leq}, \eqref{e:S3-leq}.
\end{proof}

\begin{proof}[Proof of Theorem \ref{suffthm}]
 By a result of Hahlomaa, see \cite[p.123]{hahlomaa} it suffices to show that for some $\alpha > 0$,
   \begin{align*}
    \iiint_{\Sigma(\alpha) \cap B(p,R)^3} c^2(y_1,y_2,y_3)~dy_1 ~dy_2 ~dy_3 \lesssim R, \qquad \forall p \in E, R > 0.
  \end{align*}
Hence by \eqref{e:gamma-bound}, it is enough to prove that for some $\alpha > 0$.
  \begin{align}
    \iiint_{\Sigma(\alpha) \cap B(p,R)^3} \frac{\gamma_1(y_1,y_2,y_3)^2 \gamma_2(y_1,y_2,y_3)^2}{\diam(\{y_1,y_2,y_3\})^2} ~dy_1 ~dy_2 ~dy_3 \lesssim R, \qquad \forall p \in E, R > 0. \label{e:want-gamma-bound}
  \end{align}
 By our assumption for all $\ve>0$ and  every $f \in L^2(E)$,
 \begin{equation}
 \label{l2bd}
 \|T_2^\ve f \|_{L^2(E)} \lesssim \|f\|_{L^2(E)}.
 \end{equation}
Let $p \in E$ and $R>0$. Applying \eqref{l2bd} to $f=\chi_{B(p,R)}$ we get that there exists some $C \geq 0$ so that for every $\ve > 0$,
  \begin{align*}
    \int_{E \cap B(p,R)} \int_{E \cap B(p,r) \cap B(y_1,\ve)^c} \frac{NH(y_1^{-1}y_2)^2}{d(y_1,y_2)^3} ~dy_2 \int_{E \cap B(p,r) \cap B(y_1,\ve)^c} \frac{NH(y_1^{-1}y_3)^2}{d(y_1,y_3)^3} ~dy_3 ~dy_1 \leq C R.
  \end{align*}
  \begin{align*}
U_\ve &=\{(y_1,y_2,y_3) \in \Sigma(\alpha) \cap B(p,R)^3: d(y_1,y_2)>\ve, d(y_1,y_3)>\ve \},\\
V_\ve&=\{(y_1,y_2,y_3) \in \Sigma(\alpha) \cap B(p,R)^3: d(y_1,y_2)>\ve, d(y_1,y_3)>\ve, d(y_2,y_3)>\ve \},\\
\end{align*}
  It then easily follows from Fubini (remember that all the terms in the integrand are positive) that
  \begin{align}
   \label{mmvfub1}
    \iiint_{U_\ve} \frac{NH(y_1^{-1}y_2)^2 NH(y_1^{-1}y_3)^2}{\diam(\{y_1,y_2,y_3\})^6} ~dy_1 ~dy_2 ~dy_3 \leq C R.
  \end{align}
  Therefore,
  \begin{equation}
  \begin{split}
  \label{mmvfub2}
  C R& \geq  \iiint_{V_\ve} \frac{NH(y_1^{-1}y_2)^2 NH(y_1^{-1}y_3)^2}{\diam(\{y_1,y_2,y_3\})^6} ~dy_1 ~dy_2 ~dy_3 \\
  &\quad\quad\quad\quad + \iiint_{U_\ve \stm V_\ve} \frac{NH(y_1^{-1}y_2)^2 NH(y_1^{-1}y_3)^2}{\diam(\{y_1,y_2,y_3\})^6} ~dy_1 ~dy_2 ~dy_3.
 \end{split}
  \end{equation}
Using the upper regularity $\mu$ it is not difficult to show that 
\begin{equation}
\label{mmvfub3}
\iiint_{U_\ve \stm V_\ve} \frac{NH(y_1^{-1}y_2)^2 NH(y_1^{-1}y_3)^2}{\diam(\{y_1,y_2,y_3\})^6} ~dy_1 ~dy_2 ~dy_3 \lesssim_\xi R.
\end{equation}
Using \eqref{mmvfub1},\eqref{mmvfub2}, \eqref{mmvfub3} and letting $\ve \ra 0$ we deduce that
  \begin{align*}
    \iiint_{\Sigma(\alpha) \cap B(p,R)^3} \frac{NH(y_1^{-1}y_2)^2 NH(y_1^{-1}y_3)^2}{\diam(\{y_1,y_2,y_3\})^6} ~dy_1 ~dy_2 ~dy_3 \leq C R.
  \end{align*}
By permuting variables, we get
  \begin{align}
    \iiint_{\Sigma(\alpha) \cap B(p,R)^3} \sum_{\sigma \in S_3} \frac{NH(y_{\sigma(1)}^{-1}y_{\sigma(2)})^2 NH(y_{\sigma(1)}^{-1}y_{\sigma(3)})^2}{\diam(\{y_1,y,y_3\})^6} ~dy_1 ~dy_2 ~dy_3 \leq 6CR. \label{e:have-permute-bound}
  \end{align}
  If $(y_1,y_2,y_3) \in \Sigma(\alpha)$, then it follows easily that
  \begin{align}
    \frac{\gamma_1(y_1,y_2,y_3)^2 \gamma_2(y_1,y_2,y_3)^2}{\diam(\{y_1,y_2,y_3\})^2} &\lesssim \max_{\sigma \in S_3} \frac{NH(y_{\sigma(1)}^{-1}y_{\sigma(2)})^2 NH(y_{\sigma(1)}^{-1}y_{\sigma(3)})^2}{\diam(\{y_1,y_2,y_3\})^6} \notag \\
    &\leq \sum_{\sigma \in S_3} \frac{NH(y_{\sigma(1)}^{-1}y_{\sigma(2)})^2 NH(y_{\sigma(1)}^{-1}y_{\sigma(3)})^2}{\diam(\{y_1,y_2,y_3\})^6}, \label{e:have-sigma-permute} 
  \end{align}
  where the constant multiple implicit in the first inequality depends on $\alpha$.  We then get \eqref{e:want-gamma-bound} from \eqref{e:have-permute-bound} and \eqref{e:have-sigma-permute}.
\end{proof}


\begin{bibdiv}
\begin{biblist}

\bib{blu}{book}{
AUTHOR = {Bonfiglioli, A.},
AUTHOR={Lanconelli, E.},
AUTHOR={Uguzzoni, F.},
     TITLE = {Stratified {L}ie groups and potential theory for their
              sub-{L}aplacians},
    SERIES = {Springer Monographs in Mathematics},
 PUBLISHER = {Springer, Berlin},
      YEAR = {2007},
}

\bib{CMPT}{article}{
 AUTHOR = {Chousionis, V.},
 AUTHOR={Mateu, J. },
 AUTHOR={Prat, L.},
 AUTHOR={Tolsa, X.},
     TITLE = {Calder\'on-{Z}ygmund kernels and rectifiability in the plane},
   JOURNAL = {Adv. Math.},
  FJOURNAL = {Advances in Mathematics},
    VOLUME = {231},
      YEAR = {2012},
    NUMBER = {1},
     PAGES = {535--568}
     }
     
        \bib{CM1}{article}{
 AUTHOR = {Chousionis, V.},
 AUTHOR={Mattila, P. },
     TITLE = {Singular integrals on {A}hlfors-{D}avid regular subsets of the
              {H}eisenberg group},
   JOURNAL = {J. Geom. Anal.},
  FJOURNAL = {Journal of Geometric Analysis},
    VOLUME = {21},
      YEAR = {2011},
    NUMBER = {1},
     PAGES = {56--77},
     }
     \bib{CM}{article}{
 AUTHOR = {Chousionis, V.},
 AUTHOR={Mattila, P. },
        TITLE = {Singular integrals on self-similar sets and removability for
              {L}ipschitz harmonic functions in {H}eisenberg groups},
   JOURNAL = {J. Reine Angew. Math.},
  FJOURNAL = {Journal f\"ur die Reine und Angewandte Mathematik. [Crelle's
              Journal]},
    VOLUME = {691},
      YEAR = {2014},
     PAGES = {29--60},
     }
     
       \bib{CMT}{article}{
 AUTHOR = {Chousionis, V.},
  AUTHOR = {Magnani, V.},
   AUTHOR = {Tyson, J. T.},
       TITLE = {Removable sets for {L}ipschitz harmonic functions on {C}arnot
              groups},
   JOURNAL = {Calc. Var. Partial Differential Equations},
  FJOURNAL = {Calculus of Variations and Partial Differential Equations},
    VOLUME = {53},
      YEAR = {2015},
    NUMBER = {3-4},
     PAGES = {755--780},
     }
     
        \bib{CFO}{article}{
 AUTHOR = {Chousionis, V.},
 AUTHOR={Fassler, K. },
 AUTHOR={Orponen, T.},
     TITLE = {Intrinsic Lipschitz graphs and vertical $\beta$-numbers in the Heisenberg group ''}
   JOURNAL = {preprint},
NOTE={\url{http://arxiv.org/abs/1606.07703}}
     }
     
        \bib{chun}{article}{
 AUTHOR = {Chunaev, P.},
     TITLE = {A new family of singular integral operators whose L2-boundedness implies rectifiability},
   JOURNAL = {preprint}
 NOTE={\url{   http://arxiv.org/abs/1601.07319}}
 
     }
     
     \bib{chumt}{article}{
 AUTHOR = {Chunaev, P.},
 AUTHOR={Mateu, J. },
 AUTHOR={Tolsa, X.},
     TITLE = {Singular integrals unsuitable for the curvature method whose $L^2$-boundedness still implies rectifiability },
   JOURNAL = {preprint},
NOTE={\url{http://arxiv.org/abs/1607.07663}}
     }

\bib{christ}{article}{
  title = {A $T(b)$ theorem with remarks on analytic capacity and the Cauchy integral},
  author = {Christ, M.},
  journal = {Colloq. Math.},
  volume = {60/61},
  number = {2},
  pages = {601-628},
  year = {1990},
}

\bib{Dsuff}{article}{
  title = {Morceaux de graphes Lipschitziens et int\'egrales singuli\`eres sur un surface},
  author = {David, G.},
  journal = {Rev. Mat. Iberoam.},
  volume = {4},
  number = {1},
  pages = {73-114},
  year = {1988},
}

\bib{david-wavelets}{book}{
  title = {Wavelets and singular integrals on curves and surfaces},
  author = {David, G.},
  series = {Lecture Notes in Mathematics},
  volume = {1465},
  publisher = {Springer-Verlag},
  year = {1991},
  }
  
\bib{Da3}{article}{
    AUTHOR = {David, G.},
     TITLE = {Unrectifiable {$1$}-sets have vanishing analytic capacity},
   JOURNAL = {Rev. Mat. Iberoamericana},
  FJOURNAL = {Revista Matem\'atica Iberoamericana},
    VOLUME = {14},
      YEAR = {1998},
    NUMBER = {2},
     PAGES = {369--479},
}
	
  
\bib{DM}{article}{
  AUTHOR = {David, G.},
 AUTHOR={Mattila, Pertti},
     TITLE = {Removable sets for {L}ipschitz harmonic functions in the
              plane},
   JOURNAL = {Rev. Mat. Iberoamericana},
    VOLUME = {16},
      YEAR = {2000},
    NUMBER = {1},
     PAGES = {137--215},
 }
 
 \bib{denghan}{book}{
  AUTHOR = {Deng, D.}
  AUTHOR= {Han, Y.},
     TITLE = {Harmonic analysis on spaces of homogeneous type},
    SERIES = {Lecture Notes in Mathematics},
    VOLUME = {1966},
 PUBLISHER = {Springer-Verlag, Berlin},
      YEAR = {2009},
     PAGES = {xii+154},
     }
     
     
 \bib{EiV}{article}{
  AUTHOR = {Eiderman, V.},
 AUTHOR={Volberg, A},
     TITLE = {Nonhomogeneous harmonic analysis: 16 years of development},
   JOURNAL = {Uspekhi Mat. Nauk},
    VOLUME = {68},
      YEAR = {2013},
    NUMBER = {6(414)},
     PAGES = {3--58},
 }
 
 
 
 \bib{FolS}{book}{
  AUTHOR = {Folland, G. B.}
  AUTHOR={Stein, E. M.},
     TITLE = {Hardy spaces on homogeneous groups},
    SERIES = {Mathematical Notes},
    VOLUME = {28},
 PUBLISHER = {Princeton University Press, Princeton, N.J.; University of
              Tokyo Press, Tokyo},
      YEAR = {1982},
}

  \bib{hahlomaa-2}{article}{
     AUTHOR = {Hahlomaa, I.},
     TITLE = {Menger curvature and Lipschitz parameterizations in metric spaces},
   JOURNAL = {Fund. Math.},
    VOLUME = {185},
    NUMBER = {2},
      YEAR = {2005},
     PAGES = {143--169},
     }

  \bib{hahlomaa}{article}{
     AUTHOR = {Hahlomaa, I.},
     TITLE = {Curvature integral and {L}ipschitz parametrization in
              1-regular metric spaces},
   JOURNAL = {Ann. Acad. Sci. Fenn. Math.},
    VOLUME = {32},
      YEAR = {2007},
    NUMBER = {1},
     PAGES = {99--123},
     }
 
    \bib{huo}{article}{
    AUTHOR = {Huovinen, P.},
     TITLE = {A nicely behaved singular integral on a purely unrectifiable
              set},
   JOURNAL = {Proc. Amer. Math. Soc.},
  FJOURNAL = {Proceedings of the American Mathematical Society},
    VOLUME = {129},
      YEAR = {2001},
    NUMBER = {11},
     PAGES = {3345--3351},
}


     \bib{jnaz}{article}{
 AUTHOR = {Jaye, B.},
 AUTHOR={Nazarov, F. },
     TITLE = {Three revolutions in the kernel are worse than one},
   JOURNAL = {preprint}
   NOTE={\url{http://arxiv.org/abs/1307.3678}}
     }

     
     \bib{li-schul-1}{article}{
       AUTHOR = {Li, S.},
       AUTHOR={Schul, R.},
     TITLE = {An upper bound for the length of a traveling salesman path in
              the {H}eisenberg group},
   JOURNAL = {Rev. Mat. Iberoam.},
  FJOURNAL = {Revista Matem\'atica Iberoamericana},
    VOLUME = {32},
      YEAR = {2016},
    NUMBER = {2},
     PAGES = {391--417},
     }
     
      \bib{li-schul-2}{article}{
        AUTHOR = {Li, S.},
       AUTHOR={Schul, R.},
     TITLE = {The traveling salesman problem in the {H}eisenberg group:
              upper bounding curvature},
   JOURNAL = {Trans. Amer. Math. Soc.},
  FJOURNAL = {Transactions of the American Mathematical Society},
    VOLUME = {368},
      YEAR = {2016},
    NUMBER = {7},
     PAGES = {4585--4620}
     }


\bib{MMV}{article}{
 AUTHOR = {Mattila, P.},
 AUTHOR={Melnikov, M.},
 AUTHOR={Verdera, J.},
     TITLE = {The {C}auchy integral, analytic capacity, and uniform
              rectifiability},
   JOURNAL = {Ann. of Math. (2)},
  FJOURNAL = {Annals of Mathematics. Second Series},
    VOLUME = {144},
      YEAR = {1996},
    NUMBER = {1},
     PAGES = {127--136}
}
\bib{MeV}{article}{
AUTHOR = {Melnikov, M.},
 AUTHOR={Verdera, J.},
     TITLE = {A geometric proof of the {$L^2$} boundedness of the
              {C}auchy integral on {L}ipschitz graphs},
   JOURNAL = {Internat. Math. Res. Notices},
  FJOURNAL = {International Mathematics Research Notices},
      YEAR = {1995},
    NUMBER = {7},
     PAGES = {325--331}
     }
    \bib{NToV1}{article}{ 
    AUTHOR = {Nazarov, F.},
 AUTHOR={Tolsa, X.},
 AUTHOR={Volberg, A. },
     TITLE = {On the uniform rectifiability of {AD}-regular measures with
              bounded {R}iesz transform operator: the case of codimension 1},
   JOURNAL = {Acta Math.},
  FJOURNAL = {Acta Mathematica},
    VOLUME = {213},
      YEAR = {2014},
    NUMBER = {2},
     PAGES = {237--321}
     }
     
    \bib{NToV2}{article}{ 
    AUTHOR = {Nazarov, F.},
 AUTHOR={Tolsa, X.},
 AUTHOR={Volberg, A. },
     TITLE = {The {R}iesz transform, rectifiability, and removability for
              {L}ipschitz harmonic functions},
   JOURNAL = {Publ. Mat.},
  FJOURNAL = {Publicacions Matem\`atiques},
    VOLUME = {58},
      YEAR = {2014},
    NUMBER = {2},
     PAGES = {517--532},
}

\bib{Pan1}{article}{
AUTHOR = {Pansu, P.},
     TITLE = {G\'eom\'etrie du groupe d' Heisenberg} , 
   JOURNAL = {Th\'ese pour le titre de Docteur de 3\'eme cycle,  Universit\'e Paris VII},
 
      YEAR = {1982},
  
}

\bib{Pan2}{article}{
AUTHOR = {Pansu, P.},
     TITLE = {Une in\'egalit\'e isop\'erim\'etrique sur le groupe de
              {H}eisenberg},
   JOURNAL = {C. R. Acad. Sci. Paris S\'er. I Math.},
    VOLUME = {295},
      YEAR = {1982},
    NUMBER = {2},
     PAGES = {127--130},
}

\bib{Pansu}{article}{
 AUTHOR = {Pansu, P.},
     TITLE = {M\'etriques de {C}arnot-{C}arath\'eodory et quasiisom\'etries
              des espaces sym\'etriques de rang un},
   JOURNAL = {Ann. of Math. (2)},
  FJOURNAL = {Annals of Mathematics. Second Series},
    VOLUME = {129},
      YEAR = {1989},
    NUMBER = {1},
     PAGES = {1--60},
}
\bib{St}{book}{
    AUTHOR = {Stein, E. M.},
     TITLE = {Harmonic analysis: real-variable methods, orthogonality, and
              oscillatory integrals},
    SERIES = {Princeton Mathematical Series},
    VOLUME = {43},
      NOTE = {With the assistance of Timothy S. Murphy,
              Monographs in Harmonic Analysis, III},
 PUBLISHER = {Princeton University Press, Princeton, NJ},
      YEAR = {1993},
     PAGES = {xiv+695}
}



\bib{Tolsa}{article}{
   AUTHOR = {Tolsa, X.},
     TITLE = {Uniform rectifiability, {C}alder\'on-{Z}ygmund operators with
              odd kernel, and quasiorthogonality},
   JOURNAL = {Proc. Lond. Math. Soc. (3)},
  FJOURNAL = {Proceedings of the London Mathematical Society. Third Series},
    VOLUME = {98},
      YEAR = {2009},
    NUMBER = {2},
     PAGES = {393--426},
}

\bib{tolsabook}{book}{
    AUTHOR = {Tolsa, X.},
     TITLE = {Analytic capacity, the {C}auchy transform, and non-homogeneous
              {C}alder\'on-{Z}ygmund theory},
    SERIES = {Progress in Mathematics},
    VOLUME = {307},
 PUBLISHER = {Birkh\"auser/Springer, Cham},
      YEAR = {2014},
     PAGES = {xiv+396}
}
	


	







\end{biblist}
\end{bibdiv}
\end{document}